\documentclass{amsart}

\usepackage{amscd}
\usepackage{amsfonts}
\usepackage{amsmath}
\usepackage{amssymb}
\usepackage{amsthm}
\usepackage{amsxtra}
\usepackage{array}
\usepackage{fancybox}
\usepackage{fancyhdr}
\usepackage{graphics}
\usepackage{latexsym}
\usepackage{mathrsfs}
\usepackage{upref}
\usepackage{verbatim}
\usepackage[arrow, matrix, curve]{xy}

\newtheorem{lemma}{Lemma}[section]
\newtheorem{theorem}[lemma]{Theorem}
\newtheorem{proposition}[lemma]{Proposition}
\newtheorem{corollary}[lemma]{Corollary}
\newtheorem{definition}[lemma]{Definition}

\numberwithin{equation}{section}
\newtheorem*{acknowledgement}{Acknowledgment}

\newcommand{\p}{\mathfrak{p}}
\newcommand{\Q}{\mathfrak{Q}}
\newcommand{\q}{\mathfrak{q}}

\makeindex

\begin{document}
\title[Toward an efficient algorithm]{
Toward an efficient algorithm for deciding the vanishing
of local cohomology modules in prime characteristic}
\author{Yi Zhang}
\address{Dept. of Mathematics, University of Illinois at Urbana-Champaign,
Urbana, IL 61801}
\email{zhang397@illinois.edu}
\thanks{NSF support through grant DMS-0701127 is gratefully acknowledged.}
\date{}

\begin{abstract}
Let $R=k[x_1,\dots,x_n]$ be a ring of polynomials over a field $k$ of characteristic $p>0$. There is an algorithm due to Lyubeznik for deciding the vanishing of local cohomology modules $H^i_I(R)$ where $I\subset R$ is an ideal. This algorithm has not been implemented because its complexity grows very rapidly with the growth of $p$ which makes it impractical. In this paper we produce a modification of this algorithm that consumes a modest amount of memory.
\end{abstract}
\maketitle

\section*{Introduction}

Since A. Grothendieck introduced local cohomology in 1961 \cite{aG67}, people
have been interested in the structure of local cohomology modules. Let $R$ be a
commutative ring, let $I\subset R$ be an ideal and let $M$ be an $R$-module. As
a rule, local cohomology modules $H^t_I (M )$ are not finitely
generated even if the module $M$ is. So it is very
difficult to tell whether these local cohomology modules vanish or not, and to
this day, no algorithm has been found to decide their vanishing. 

However, in the case that $R=k[x_1,\dots, x_n]$ is the ring of polynomials in a
finite number of variables over a field $k$ and $M=R,$ two completely different
algorithms are known, one in characteristic 0 \cite{uW99}, the other in
characteristic $p>0$ \cite[Remark 2.4]{gL97}. The characteristic 0
algorithm uses ideas from the theory of $D$-modules, while the characteristic
$p>0$ algorithm uses ideas from the theory of $F$-modules. The characteristic 0
algorithm has been implemented and is part of the computer package
``$D$modules" for Macaulay 2 \cite{GS}. The characteristic $p>0$ algorithm has
not been implemented since its complexity grows very rapidly with the growth of
$p$ which makes it impractical. 

More precisely, let $R=\mathbb Z[x_1,\dots,x_n]$, let $f_1,\dots,f_s\in R$ be
polynomials in variables $x_1,\dots,x_n$ with integer coefficients, and let
$I=(f_1,\dots,f_s)\subset R$ be the ideal they generate. For a prime
integer $p>0,$ let $\bar {\mathbb Z}=\mathbb Z/p\mathbb Z$, let $\bar R=\bar
{\mathbb Z}[x_1,\dots,x_n]$, let $\bar f_i\in \bar R$ be the polynomial obtained
from $f_i$ by reducing its coefficients modulo $p,$ and let $\bar I$
be the ideal of $\bar R$ generated by $\bar f_1,\dots, \bar f_s$. We keep this notation for the rest of the paper.

The algorithm from \cite[Remark 2.4]{gL97} for deciding the vanishing of the local cohomology module $H^t_{\bar {I}}(\bar R)$  
involves computations with the ideal $\bar I^{[p]}$ generated by the $p$-th
powers of $\bar f_1,\dots, \bar f_s$. The complexity of these computations grows
very rapidly with the growth of $p$ because the degrees of the polynomials $\bar
f_i^p$ that generate the ideal $\bar{I}^{[p]}$ grow linearly and the amount
of memory required to perform Gr\"obner bases calculations grows exponentially in the degrees of the generators \cite{MM82}.

In this paper, we present a modification of the algorithm  from
\cite[Remark 2.4]{gL97}. The amount of memory our modification  consumes
grows only linearly with the growth of $p.$ Unfortunately, this is not enough to
produce a fully practical algorithm since the number of operations still grows
very rapidly with the growth of $p$, an extraordinary amount of time may be
required to complete the calculation. Nevertheless, at least available memory is unlikely to be exhausted before the calculation is completed. 

We view our result as an important step in a search for a fully practical algorithm. For our result shows that at least in terms of required memory, there is no obstacle to finding such an algorithm.

\section{Preliminaries}

Recall that $R=\mathbb Z[x_1,\dots, x_n]$ is a ring of polynomials over the integers, $p\in \mathbb Z$ is a prime number and $\bar R=R/pR=(\mathbb Z/p\mathbb Z)[x_1,\dots,x_n]$. Local cohomology modules $H^i_I(R)$ have a structure of $F$-finite modules in the sense of \cite{gL97}. In this section we review the algorithm from \cite[Remark 2.4]{gL97} for deciding the vanishing of $F$-finite modules and discuss some ingredients of our modification of this algorithm. 

Given an integer $\ell,$ the
$\ell$-fold Frobenius homomorphism is
$F^\ell: \bar R_s\xrightarrow{r\mapsto r^{p^\ell}}\bar R_t,$ where $\bar R_s$ and $\bar R_t$ are
copies of $\bar R$ (the subscripts stand for source and target). There are two 
associated functors, namely, the push-forward
$$F_*^\ell:\bar R_t\text{-mod}\rightarrow \bar R_s\text{-mod}$$ which is just the
restriction of scalars functor (i.e. $F_*^{\ell}(M)$, for an $\bar R_t$-module $M$ is $M$ viewed as an $\bar R_s$-module via $F^\ell$) and the pull-back
$$F^{*^\ell}:\bar R_s\text{-mod}\rightarrow \bar R_t\text{-mod}$$ such that
$F^{*^\ell}(N)=\bar R_t\otimes_{\bar R_s} N$ and $F^{*^\ell}(N \xrightarrow{\lambda}
N')=(\bar R_t\otimes_{\bar R_s} N \xrightarrow{\bar R_t\otimes_{\bar R_s} \lambda}
\bar R_t\otimes_{\bar R_s} N').$ Normally one suppresses the subscripts and thinks of $F^{*^\ell}$ and $F_*^\ell$ as functors from $\bar R$-modules to $\bar R$-modules: $$F^{*^\ell}, F^\ell_*:\bar R\text{-mod}\rightarrow \bar R\text{-mod}.$$ 

For every $R$-module $M$ we set $\bar M=M/pM$; every $\bar R$-module is of the form $\bar M$ for some $R$-module $M$. Let an $\bar R$-module $\mathcal M$ be the limit of the inductive system 
\begin{equation}\label{M:GenMorp}
\bar M \stackrel{\beta}{\to}  F^*(\bar M) \stackrel{F^{*}(\beta)}{\to}
F^{*^2}(\bar M)\stackrel{F^{*^2}(\beta)}{\to}\dots
\end{equation} where $\bar M$ is a finitely generated $\bar R$-module and $\beta:\bar M\to F^*(\bar M)$ is an $\bar R$-module homomorphism. The module $\mathcal M$ is the underlying $\bar R$-module of an $F$-finite module which is defined as a pair $(\mathcal M,\theta)$ where $\theta:\mathcal M\to F^*(\mathcal M)$ is an $\bar R$-module isomorphism \cite[Definitions 1.1 and 1.9]{gL97}. The isomorphism $\theta$ is not going to play any role in this paper because we are interested only in the vanishing of this $F$-finite module $(\mathcal M,\theta)$ which by definition means the vanishing of the underlying $\bar R$-module $\mathcal M$. For this reason we omit the definition of $\theta$. By a slight abuse of terminology we call $\mathcal M$ itself the $F$-finite module generated by $\beta:\bar M\to F^*(\bar M)$ (this map is called a generating morphism of $\mathcal M$).

The following proposition underlies the algorithm from \cite[Remark 2.4]{gL97} for deciding the vanishing of $F$-finite modules. 

\begin{proposition}\label{PA} \cite[Proposition 2.3]{gL97}
Suppose $\mathcal M$ is an $F$-finite module and let $\beta:\bar M\to F^*(\bar M)$ be a
generating morphism of $\mathcal M$ such that $\bar M$ is a finitely generated $\bar R$-module.
Let $\beta_j:\bar M\to F^{*^j}(\bar M)$ be the composition 
$$\bar M\xrightarrow{\beta} F^*(\bar M) \xrightarrow{F^*(\beta)} \cdots \xrightarrow{F^{*^{j-1}}(\beta)} F^{*^j}(\bar M).$$
Then:
\begin{enumerate}
\item[(a)] The ascending chain $\ker\beta_1\subset \ker\beta_2\subset \cdots$ of submodules of $\bar M$ eventually stabilizes.
Let $C \subset \bar M$ be the common value of $\ker \beta_i$ for sufficiently big $i.$
\item[(b)] If $r$ is the first integer such that $\ker \beta_r = \ker \beta_{r+1},$ then $\ker \beta_r =C.$
\item[(c)] $\mathcal M=0$ if and only if $\bar M=C$, i.e., $\beta_r$ is the zero map.
\end{enumerate} 
\end{proposition}

This leads to an algorithm for deciding whether the $F$-finite module $\mathcal M$ generated by $\beta: \bar M\to F^*(\bar M)$ vanishes. We quote \cite[Remark 2.4]{gL97}:
\medskip

{\it [F]or
each integer $j=1,2,3,\dots$ one should compute the kernel of $\beta_j,$ and compare it with
the kernel of $\beta_{j-1},$ until one finds $r$ such that $\ker \beta_r= \ker
\beta_{r-1}.$ One then should check whether $\ker \beta_r$ and $\bar M$ coincide. The
$F$-finite module in question is zero if and only if they do coincide. If $R$ is
a polynomial ring in
several variables over a field, these operations are implementable on a computer
by means
of standard software like Macaulay.     }

\medskip

However, a practical implementation of this algorithm faces difficulties. Namely, to compute ker$\beta_j$ one has to be able to decide whether $\beta_j(m)\in F^{*^j}(\bar M)$, for some $m\in \bar M$, vanishes. For example, if $\bar M$ is principal, i.e., $\bar M=R/\mathfrak a$, then $F^{*^j}(\bar M)=R/\mathfrak a^{[p^j]}$ where $\mathfrak a^{[p^j]}$ is the ideal generated by the $p^j$-th powers of the generators of $\mathfrak a$. Thinking of $\beta_j(m)$ as an element of $R$ one has to decide whether $\beta_i(m)\in \mathfrak a^{[p^i]}$. If $\mathfrak a$ is generated by polynomials of degrees $d_1,\dots, d_s$,  then $\mathfrak a^{[p^j]}$ is generated by polynomials of degrees $d_1p^j,\dots, d_sp^j$. These are huge, even for modest values of $p$ and $j$. Deciding membership in an ideal generated by polynomials of huge degrees consumes a prohibitive amount of memory. 

Recall that $f_1,\dots, f_s\in R=\mathbb Z[x_1,\dots, x_n]$ are polynomials with integer coefficients, $I=(f_1,\dots, f_s)\subset R$ is the ideal they generate, $\bar f_j\in \bar R=R/pR$ is obtained from $f_j$ by reducing its coefficients modulo $p$ and $\bar I=(\bar f_1,\dots, \bar f_s)\subset \bar R$ is the ideal generated by $\bar f_1,\dots, \bar f_s\in \bar R$. Every local cohomology module $H^i_{\bar I}(\bar R)$ acquires a structure of $F$-finite module as follows. Let $K^\bullet(\bar R;\bar f_1\dots, \bar f_s)$ be the Koszul cocomplex $$0\to K^0(\bar R;\bar f_1,\dots, \bar f_s)\stackrel{d^0}{\to} K^1(\bar R;\bar f_1,\dots, \bar f_s)\stackrel{d^1}{\to}  \dots\stackrel{d^{s-1}}{\to} K^s(\bar R;\bar f_1,\dots, \bar f_s)\to 0$$
where $K^t(\bar R;\bar f_1,\cdots,\bar f_s)$ is the direct sum of copies of $\bar R$ indexed by the cardinality $t$ subsets of the set $\{1,\dots, s\}$ and the differentials are defined by $$d^t(\kappa)_{v_1\dots,v_t}=\sum_\ell(-1)^\ell\kappa_{v_1,\dots\hat{v_\ell},\dots,v_t}$$ where $\kappa\in K^{t-1}(\bar R;\bar f_1,\dots, \bar f_s)$ while $d^t(\kappa)_{v_1\dots,v_t}\in \bar R_{v_1\dots,v_t}\subseteq K^{t}(\bar R;\bar f_1,\dots, \bar f_s)$ and $\kappa_{v_1,\dots\hat{v_\ell},\dots,j_t}\in \bar R_{v_1,\dots\hat{v_\ell},\dots,v_t}\subseteq K^{t-1}(\bar R;\bar f_1,\dots, \bar f_s).$

Let $\bar M$ be the $i$-th cohomology module of $K^\bullet(\bar R;\bar f_1,\dots,\bar f_s)$. The $i$-th cohomology module of $K^\bullet(\bar R;\bar f_1^{p},\dots,\bar f_s^{p})$ is $F^{*}(\bar M)$ (\cite[Remarks 1.0(e)]{gL97}) and $H^i_I(R)$ is the $F$-finite module generated by the map $\beta:\bar M\to F^{*}(\bar M)$ which is the map induced on cohomology by the chain map $$K^\bullet(\bar R;\bar f_1,\dots,\bar f_s)\stackrel{\beta^\bullet}{\to} F^*(K^\bullet(\bar R;\bar f_1,\dots,\bar f_s))\cong K^\bullet(\bar R;\bar f_1^{p},\dots,\bar f_s^{p})$$ which is defined as follows: the chain map $\beta^\bullet$ sends $\bar R_{v_1,\dots,j_i}\subseteq K^i(\bar R;\bar f_1,\dots,\bar f_s)$ to $\bar R_{v_1,\dots,v_i}\subseteq K^i(\bar R;\bar f_1^p,\dots,\bar f_s^p)$ via multiplication by $(\bar f_{v_1}\cdots \bar f_{v_i})^{p-1}$.

In this paper we produce a modification of the algorithm from \cite[Remark 2.4]{gL97} for deciding the vanishing of the $F$-finite module $H^i_{\bar I}(\bar R)$. This modification avoids deciding membership in an ideal generated by polynomials of huge degrees and as a result it requires only a modest amount of memory. We explain the idea behind this modification after the following proposition. 

\begin{proposition}\label{modp}
Let $M$ be the $i$-th cohomology module of the Koszul cocomplex $K^\bullet(R;f_1,\dots,f_s)$. For all but finitely many prime integers $p,$ the $i$-th cohomology module of the Koszul cocomplex $K^\bullet(\bar R;\bar f_1,\dots,\bar f_s)$ is $\bar M=M/pM$.
\end{proposition}
\begin{proof} 
The cocomplex $K^\bullet(\bar R;\bar f_1,\dots,\bar f_s)$ is just $\bar {\mathbb Z}\otimes_{\mathbb Z}K^\bullet(R;f_1,\dots,f_s)$ where $\bar{\mathbb Z}=\mathbb Z/p\mathbb Z$. Since $K^j(R;f_1,\dots,f_s)$ is a finitely generated $R$-module for all $j$, by the generic freeness lemma (\cite[Lemma 8.1]{HR74}) there is $\delta\in \mathbb Z$ such that upon inverting $\delta$ the images and the kernels of the differentials in the resulting cocomplex $K^\bullet(R_\delta;f_1,\dots,f_s)$ as well as the cohomology modules of this cocomplex are free $\mathbb Z_\delta$-modules. Hence for every prime integer $p$ that does not divide $\delta$, the $i$-th cohomology module of $\bar {\mathbb Z}\otimes_{\mathbb Z}K^\bullet(R_\delta;f_1,\dots,f_s)\cong\bar {\mathbb Z}\otimes_{\mathbb Z}K^\bullet(R;f_1,\dots,f_s)=K^\bullet(\bar R;\bar f_1,\dots, \bar f_s)$ is $\bar M$.
\end{proof}

It is worth pointing out that the proof of the generic freeness lemma \cite[Lemma 8.1]{HR74} makes the integer $\delta$ algorithmically computable, given the polynomials $f_1,\dots, f_s$. We are leaving the details to the interested reader.

Now we are ready to discuss the idea behind our modification of the algorithm from \cite[Remark 2.4]{gL97} for deciding the vanishing of $H^i_{\bar I}(\bar R)$. Let $M$ be the $i$-th cohomology module of the Koszul cocomplex $K^\bullet(R;f_1,\dots,f_s)$. For every prime integer $p$ such that the $i$-th cohomology module of the Koszul cocomplex $K^\bullet(\bar R;\bar f_1,\dots,\bar f_s)$ is $\bar M=M/pM,$ let $\beta:\bar M\to F^*(\bar M)$ be a generating morphism of $H^i_{\bar I}(\bar R)$ as above and let $\beta_j:\bar M\to F^{*^j}(\bar M)$ be as in the statement of Proposition \ref{PA}. According to Proposition \ref{PA} there exists an integer $r$ such that ker$\beta_r={\rm ker}\beta_{r+1}$. In the next section, Section 2, we show that there is a computable upper bound $u$ on the minimum such integer $r$; this upper bound $u$ depends only on $M$ and is independent of the particular prime integer $p$. And in Section 3 we produce an algorithm to decide whether $\beta_j$ vanishes for fixed $j$ and $p$. It is this algorithm that consumes a modest amount of memory. But it only decides the vanishing of $\beta_j$, not whether ker$\beta_j={\rm ker}\beta_{j-1}$. It is for this reason that we need a computable upper bound $u$ (which just happens to be the same for all prime integers $p$, so $u$ has to be computed just once). According to Proposition \ref{PA}(b,c), the fact that ker$\beta_r={\rm ker}\beta_{r+1}$ for some $r\leq u$ implies that $H^i_{\bar I}(\bar R)=0$ if and only if $\beta_u=0.$ So for every prime integer $p,$ it's enough to decide whether $\beta_j=0$ for just one specific value of $j$, namely $j=u$.

\section{An upper bound on the number of steps involved in the algorithm}

In this section, $M$ is a finitely generated $R$-module where $R=\mathbb Z[x_1,\dots,x_n]$. Recall that $\bar R=\bar{\mathbb Z}[x_1,\dots,x_n]$ where $\bar{\mathbb Z}=\mathbb Z/p\mathbb Z$ and $\bar M=\bar{\mathbb Z}\otimes_{\mathbb Z}M=M/pM$. Let $\beta:\bar M\to F^*(\bar M)$ be a generating morphism of an $F$-finite module $\mathcal M$. In the preceding section, we quoted an algorithm from \cite[Remark 2.4]{gL97} that decides whether $\mathcal M=0$.
By the number of steps involved in this algorithm we mean the first integer $r$ such that ker$\beta_r=$ker$\beta_{r-1}$. The main result of this section is Corollary \ref{upperboundonr} which produces an upper bound on $r$ that depends only on $M$ (i.e. it is independent of $p$ and $\beta$).

\begin{lemma} \label{L:EVker}
Notation being as above, if $\bar M$ has finite length in the category of $\bar R$-modules, then the first integer $r$ such that ker$\beta_r=$ker$\beta_{r-1}$ satisfies the inequality $r\leq u$, where $u$ is the length of $\bar M$. In particular,
$\mathcal M=0$ if and only if ${\rm ker}\beta_u=\bar M$, i.e., $\beta_u=0$. 
\end{lemma}
\begin{proof}
Since the length of $\bar M$ is finite, the number of strict containment in the ascending chain $\ker\beta_1\subseteq \ker\beta_2\subseteq \dots$ of submodules of $\bar M$ cannot be bigger than the length of $\bar M$. Since this ascending chain stabilizes at the first integer $r$ such that ker$\beta_r=$ker$\beta_{r-1}$, this integer $r$ must be less than or equal to the length of $\bar M$. 
\end{proof}

We define the {\it universal length $u$} of a finitely generated $\bar R$-module $N$ as follows:

\begin{definition} \label{D:u}
$u(N)=\max \{
{\rm length}\ \Gamma_{\mathfrak p}(N_{\mathfrak p})\:|\:\p\in\textnormal{Ass}N\}$, where ${\rm Ass}N$ is the set of the associated primes of $N$, the torsion functor $\Gamma_\p$ is the 0-th local cohomology functor $H^0_\p(-),$ and the length is measured in the category of $R_\mathfrak p$-modules.
\end{definition}
  
\begin{corollary} \label{C:VanishingEquiv}
Notation being as above, let $u=u(\bar M).$ The first integer $r$ such that ker$\beta_r=$ker$\beta_{r-1}$ satisfies the inequality $r\leq u$. In particular,
$\mathcal M=0$ if and only if $\beta_u=0$.
\end{corollary}
\begin{proof}
By \cite[Remark
2.13]{gL97}, we have $\textnormal{Ass}\mathcal M\subseteq \textnormal{Ass}\bar M$. Hence $\mathcal M$ vanishes if and only if $\Gamma_\p(\mathcal M_\p)$ vanishes for all
$\p\in \textnormal{Ass}\bar M$.
The module $\Gamma_{\mathfrak p}(\mathcal M_{\mathfrak p})$ is the limit of the system$$\Gamma_{\mathfrak p}(\bar M_{\mathfrak p})\to F^*(\Gamma_{\mathfrak p}(\bar M_{\mathfrak p}))\to F^{*^2}(\Gamma_{\mathfrak p}(\bar M_{\mathfrak p}))\to \dots$$ obtained by applying the functor $\Gamma_{\mathfrak p}(-_{\mathfrak p})$ to (\ref{M:GenMorp}) and taking into account that the functors $F^*$ and $\Gamma_{\mathfrak p}(-_{\mathfrak p})$ commute with each other. But the module $\Gamma_{\mathfrak p}(\bar M_{\mathfrak p})$ is of finite length over the local ring $R_{\mathfrak p}$ and its length is at most $u=u(\bar M).$ So by Lemma \ref{L:EVker},  $\Gamma_{\mathfrak p}(\mathcal M_{\mathfrak p})=0$ if and only if the composition of the first $u$ maps in the above system, i.e., the map $$\Gamma_\mathfrak p(\beta_u)_\mathfrak p:\Gamma_{\mathfrak p}(\bar M_{\mathfrak p}) \to F^{*^{u+1}}(\Gamma_{\mathfrak p}(\bar M_{\mathfrak p}))$$ is zero. But the image of this map is nothing but $(\Gamma_{\mathfrak p}({\rm im}\beta_u))_\mathfrak p$. Hence $\Gamma_\p(\mathcal M_\p)=0$ if and only if $(\Gamma_{\mathfrak p}({\rm im}\beta_u))_\mathfrak p=0$.
 
It remains to show that im$\beta_u=0$ if and only if $(\Gamma_{\mathfrak p}({\rm im}\beta_u))_\mathfrak p=0$ for every $\mathfrak p\in{\rm Ass}\bar M$. This follows from the fact that im$\beta_u$ is a submodule of $F^{*^{u+1}}(\bar M)$ and therefore Ass(im$\beta_u)\subseteq{\rm Ass}F^{*^{u+1}}(\bar M)={\rm Ass}\bar M$ by \cite[Corollary 1.6]{HS93}.
\end{proof}

\begin{lemma} \label{associatedprimes}
For all but finitely many prime integers $p,$ the following hold.

(a) The associated primes of $\bar M$ are minimal primes of ideals $(p, \mathfrak p)$ as $\mathfrak p$ runs over the associated primes of $M$, and

(b) $\overline{\Gamma_{\mathfrak p}(M)}\stackrel{\rm def}{=}\Gamma_\mathfrak p(M)/p\Gamma_\mathfrak p(M)\cong \Gamma_{(p, \mathfrak p)}(\bar M)$ for every associated prime $\mathfrak p$ of $M$.
\end{lemma}
\begin{proof}
(a) The set of the associated primes of $M$ is finite and each associated prime of $M$ contains at most one prime integer $p$. Hence all but finitely many prime integers $p$ do not belong to any associated prime of $M$. Fix one such prime integer $p$. 

Let $\q$ be a prime ideal of $R$ containing the integer $p$ and associated to $\bar M$. This is the case if and only if $\bar M_{\q}\ne 0$ and depth$\bar M_{\mathfrak q}=0$. Since $p\in \mathfrak q$ does not belong to any associated prime of $M$, the prime ideal $\mathfrak q$ is not associated to $M$, i.e., depth$M_{\mathfrak q}>0$. Since $\bar M_{\mathfrak q}=M_{\mathfrak q}/pM_{\mathfrak q}$, we conclude that $\{p\}$ is a maximal $M_{\mathfrak q}$-regular sequence of length 1, i.e., depth$M_{\mathfrak q}=1$. 

Let $h={\rm dim}R_\mathfrak q={\rm height}\mathfrak q,$ then the Auslander-Buchsbaum theorem (\cite[Theorem 19.1]{hM89}) implies that the projective dimension of $M_\mathfrak q$ in the category of $R_\mathfrak q$-modules is $h-{\rm depth}M_{\mathfrak q}=h-1$.  This in turn implies that Ext$^{h-1}_{R_{\mathfrak q}}( M_{\mathfrak q}, R_{\mathfrak q})\ne 0$. Since Ext$^{h-1}_{R_{\mathfrak q}}( M_{\mathfrak q}, R_{\mathfrak q})={\rm Ext}^{h-1}_R( M, R)_\mathfrak q$, we conclude that the prime ideal $\mathfrak q$ is in the support of Ext$^{h-1}_R( M, R)$.

If $\Q$ is a prime ideal of height $<h-1,$ then $R_\mathfrak Q$ is regular and of dimension $<h-1,$ hence $\text{Ext}^{h-1}_R( M, R)_\Q=\text{Ext}^{h-1}_{R_\Q}(M_\Q,R_\Q)=0$. Therefore every minimal prime of the $R$-module Ext$^{h-1}_R( M, R)$ has height at least $h-1$. 

The height $h-1$ minimal primes of Ext$^{h-1}_R( M, R)$ are precisely the associated primes of $M$ of height $h-1$. Indeed, Ext$^{h-1}_R( M, R)_\mathfrak p={\rm Ext}^{h-1}_{R_\mathfrak p}(M_\mathfrak p,R_\mathfrak p)\ne 0$ for a height $h-1$ prime ideal $\mathfrak p$ is equivalent by the Auslander-Buchsbaum theorem (\cite[Theorem 19.1]{hM89}) to depth$M_\mathfrak p=0$, i.e., $\mathfrak p$ being associated to $M$. 

If $\mathfrak q$ contains a minimal prime $\mathfrak p$ of Ext$^{h-1}_R( M, R)$ of height $h-1$, then $\mathfrak q$, being of height $h$ and containing $p\not\in\mathfrak p$, is a minimal prime over the ideal $(p,\mathfrak p)$. 

If $\mathfrak q$ does not contain a minimal prime of Ext$^{h-1}_R( M, R)$ of height $h-1$, then $\mathfrak q$, being of height $h$ and in the support of Ext$^{h-1}_R( M, R)$, is itself a minimal prime of Ext$^{h-1}_R( M, R)$ because every minimal prime of Ext$^{h-1}_R( M, R)$ has height at least $h-1$. 

Thus if a prime integer $p$ does not belong to any associated prime of $M$ and does not belong to any associated prime of Ext$^{h-1}_R( M, R)$ of height $h$, as $h$ runs over all integers $\leq {\rm dim}R$, then every associated prime of $\bar M$ is a minimal prime over the ideal $(p,\mathfrak p)$ where $\mathfrak p$ is an associated prime of $M$. Since the set of the associated primes of $M$ and the set of the associated primes of Ext$^{h-1}_R( M, R)$ of height $h$ are finite, all but finitely many prime integers $p$ have this property. This proves (a).

(b) The modules in the short exact sequence $0\to \Gamma_\mathfrak p(M)\to M\to M/\Gamma_\mathfrak p(M)\to 0$ are finitely generated over $R$ and $R$ is a finitely generated $\mathbb Z$-algebra. Hence by the generic freeness lemma (\cite[Lemma 8.1]{HR74}) there is an integer $\gamma\in \mathbb Z$ such that $\Gamma_\mathfrak p(M)_\gamma, M_\gamma$ and $(M/\Gamma_\mathfrak p(M))_\gamma$ are free $\mathbb Z_\gamma$-modules. Since the induced sequence of free $\mathbb Z_\gamma$-modules $0\to \Gamma_\mathfrak p(M)_\gamma\to M_\gamma\to (M/\Gamma_\mathfrak p(M))_\gamma\to 0$ is exact, tensoring over $\mathbb Z$ with $\mathbb Z/p\mathbb Z$ for a prime integer $p$ which does not divide $\gamma$ produces an exact sequence  $$0\to \overline{\Gamma_{\mathfrak p}(M)}\to \bar M\to \overline{M/\Gamma_\mathfrak p(M)}\to 0$$ where $\overline{\Gamma_{\mathfrak p}(M)}=\Gamma_\mathfrak p(M)/p\Gamma_\mathfrak p(M)$ and $\overline{M/\Gamma_\mathfrak p(M)}=(M/\Gamma_\mathfrak p(M))/p(M/\Gamma_\mathfrak p(M))$.

Viewing $\overline{\Gamma_{\mathfrak p}(M)}$ as a submodule of $\bar M$ and considering that every element of $\overline{\Gamma_{\mathfrak p}(M)}$ is annihilated both by $p$ and by some power of the ideal $\mathfrak p,$ we conclude that $\overline{\Gamma_{\mathfrak p}(M)}\subseteq \Gamma_{(p, \mathfrak p)}(\bar M)$. To prove (b) that this containment is actually an equality for all but finitely many $p,$ it is enough to show that $\Gamma_{(p,\mathfrak p)}(\bar M/\overline{\Gamma_{\mathfrak p}(M)})=0,$ i.e., $\Gamma_{(p,\mathfrak p)}(\overline{M/\Gamma_\mathfrak p(M)})=0$ for all but finitely many $p$ (since $\bar M/\overline{\Gamma_{\mathfrak p}(M)}\cong\overline{M/\Gamma_\mathfrak p(M)}$). And to prove this vanishing, it is enough to show that for all but finitely many $p,$ none of the minimal primes of the ideal $(p,\mathfrak p)$ are associated to $\overline{M/\Gamma_\mathfrak p(M)}$.

Let $h$ be the height of $\mathfrak p$ and let $\mathfrak p_1,\dots,\mathfrak p_s$ be the associated primes of $M/\Gamma_\mathfrak p(M)$ of height $h$. Since the heights of $\mathfrak p$ and $\mathfrak p_i$ are the same and $\mathfrak p$ is not associated to $M/\Gamma_\mathfrak p(M)$, i.e., $\mathfrak p\not =\mathfrak p_i$ for every $i$, the ideals $\mathfrak p+\mathfrak p_i$ are bigger than $\mathfrak p$ for every $i$. Hence the height of every prime ideal containing $\mathfrak p+\mathfrak p_i$ is at least $h+1$. This implies that there are only finitely many prime ideals $\mathfrak Q_1,\dots,\mathfrak Q_v$ of $R$ of height $h+1$ that contain both $\mathfrak p$ and $\mathfrak p_i$ for some $i$. 

Since only finitely many prime integers $p$ are contained in one of these prime ideals $\mathfrak Q_1,\dots,\mathfrak Q_v,$ and since the height of every minimal prime over the ideal $(p,\mathfrak p)$ is $h+1$, we conclude that for all but finitely many prime integers $p,$ no minimal prime over the ideal $(p,\mathfrak p)$ coincides with one of the $\mathfrak Q_1,\dots,\mathfrak Q_v$. That is for all but finitely many prime integers $p,$ no minimal prime over the ideal $(p,\mathfrak p)$ contains one of the ideals $\mathfrak p_1,\dots,\mathfrak p_s$.

But it follows from (a) that for all but finitely many prime integers $p,$ every associated prime of $\overline{M/\Gamma_\mathfrak p(M)}$ of height $h+1$ contains an associated prime of $M/\Gamma_\mathfrak p(M)$ of height $h$, i.e., it contains one of the ideals $\mathfrak p_1,\dots,\mathfrak p_s$. This finally shows that for all but finitely many prime integers $p,$ no minimal prime over the ideal $(p,\mathfrak p)$ is associated to $\overline{M/\Gamma_\mathfrak p(M)}$ and completes the proof of (b).
\end{proof} 

\begin{corollary}\label{maximum}
The maximum of $u(\bar M=M/pM)$, as $p$ runs through all the prime integers, is finite, where $u$ is defined in Definition \ref{D:u}.
\end{corollary}

\begin{proof}
If all associated primes of $M$ contain prime integers, then for all prime integers $p,$ except those finitely many contained in associated primes of $M,$ the module $\bar M=M/pM$ is zero. Hence $u(\bar M)=0$ for all but finitely many $p.$

Otherwise, let $\p$ be an associated prime of $M$ that does not contain any prime integer. Let $h$ be the height of $\mathfrak p$. Let $y_1,\dots,y_{n-h}\in \tilde R=R\otimes_{\mathbb Z}\mathbb Q=\mathbb Q[x_1\dots,x_n]$ be linear combinations of variables $x_1,\dots,x_n$ with coefficients in $\mathbb Q$ such that $\tilde R/\mathfrak p\tilde R$ is finite over the ring $S=\mathbb Q[y_1,\dots,y_{n-h}]$. Since the field $\mathbb Q$ is infinite, generic linear combinations will do.

For each $i,$ let $\bar x_i\in \tilde R/\mathfrak p\tilde R$ be the image of $x_i$ under the natural map $\tilde R\to \tilde R/\mathfrak p\tilde R$ and let $\bar x_i^{t_i}+s_{i,1}\bar x_i^{t_i-1}+s_{i,2}\bar x_i^{t_i-2}+\dots=0$, where $s_{i,j}\in S$, be an equation expressing integral dependence of $\bar x_i$ on $S$. The polynomials $s_{i,j}\in S$ have a finite number of coefficients in $\mathbb Q$ and $y_1,\dots,y_{n-h}$, as linear combinations of $x_1,\dots,x_n$, also have a finite number of coefficients in $\mathbb Q$. All these coefficients have a common denominator $\delta\in \mathbb Z$. Hence $y_1,\dots,y_{n-h}\in R_{\delta}=\mathbb Z_{\delta}[x_1,\dots,x_n]$ and $R_{\delta}/\p R_{\delta}$ is a finite $S_{\delta}$-module where $S_{\delta}=\mathbb Z_{\delta}[y_1,\dots, y_{n-h}]$.

Since $S_\delta$ is a subring of $R_\delta$, the module $M_\delta$ has a natural structure of $S_\delta$-module, hence so does $\Gamma_{\mathfrak p}(M_\delta)$. This is a finitely generated $R_\delta$-submodule of $M_\delta$ supported at $\mathfrak p$ and therefore annihilated by some power of $\mathfrak p$. Hence $\Gamma_{\mathfrak p}(M_\delta)$ has a finite filtration with quotients finitely generated $R_\delta/{\mathfrak p}R_\delta$-modules.   Since $R_\delta/{\mathfrak p}R_\delta$ is a finitely generated $S_\delta$-module, $\Gamma_{\mathfrak p}(M)_\delta$ is a finitely generated $S_\delta$-module.

Let $p$ be a prime integer that does not divide $\delta$, does not belong to any associated prime of $M$ and does not belong to any height $h$ minimal prime of the $R$-module Ext$^{h-1}_R( M, R)$ as $h$ runs through all integers $\leq {\rm dim}R$. This includes all but finitely many prime integers $p$. 

Since $p$ does not belong to any associated prime of $M$, the module $M$ has zero $p$-torsion. Let $S_{(p)}$ be the ring $S$ localized at the principal prime ideal $pS$. Since the ring $S_{(p)}$ is a discrete valuation ring, since the module $\Gamma_\p(M)$, being a submodule of $M$, has zero $p$-torsion and since $\Gamma_\p(M)_\delta$ is a finitely generated $S_\delta$-module, we conclude that $\Gamma_\p(M)_{(p)}\stackrel{\rm def}{=}S_{(p)}\otimes_S\Gamma_\p(M)$ is a free $S_{(p)}$-module of finite rank $\rho={\rm dim}_K(K\otimes_S\Gamma_\p(M))$ where $K$ is the fraction field of $S$. Hence the dimension of $\overline{\Gamma_{\frak p}(M)_{(p)}}\stackrel{\rm def}=\Gamma_\p(M)_{(p)}/p\Gamma_\p(M)_{(p)}$ over the residue field $\kappa$ of $S_{(p)}$ also equals $\rho$. This implies that for every minimal prime $\mathfrak q$ over the ideal $(p,\mathfrak p)$ the length of  $\overline{\Gamma_{\mathfrak p}(M)}_{\mathfrak q}$ in the category of $R_{\mathfrak q}$-modules is at most $\rho$. Clearly the integer $\rho$ is independent of the prime integer $p$. 

It follows from Lemma \ref{associatedprimes}(b) that $\overline{\Gamma_{\frak p}(M)_{(p)}}\cong S_{(p)}\otimes_S\Gamma_{(p,\mathfrak p)}(\bar M)$ for all but finitely many prime integers $p$. Hence for every minimal prime $\mathfrak q$ over the ideal $(p,\mathfrak p),$ the length of $\Gamma_{\mathfrak q}(\bar M)_{\mathfrak q}$ in the category of $\bar R_{\mathfrak q}$-modules is at most $\rho$, which is independent of $p$. But according to Lemma \ref{associatedprimes}(a), for all but finitely many $p,$ every associated prime of $\bar M$ is minimal over $(p,\mathfrak p)$ for some associated prime $\mathfrak p$ of $M$.
\end{proof}

\begin{corollary}\label{upperboundonr}
Let $u=u(M)$ be the maximum of $u(\bar M=M/pM)$, as $p$ runs through all the prime integers. Let $p$ be any prime integer, let $\beta:\bar M\to F(\bar M)$ be an $\bar R$-module homomorphism and let $\mathcal M$ be the $F$-finite module generated by $\beta$. 

(a) The first integer $r$ such that ker$\beta_r=$ker$\beta_{r-1}$ satisfies the inequality $r\leq u$.

(b) $\mathcal M=0$ if and only if $\beta_u=0$.
\end{corollary}
\begin{proof} 
This is a consequence of Corollaries \ref{C:VanishingEquiv} and \ref{maximum}. 
\end{proof}

This corollary establishes an upper bound on the number of steps involved in the algorithm (i.e. on the first integer $r$ such that ker$\beta_r$=ker$\beta_{r-1}$). This upper bound depends only on the $R$-module $M$ and is independent of the prime integer $p$ and even of the $\bar R$-module map $\beta:\bar M\to F^*(\bar M)$.

The integer $u=u(M)$ plays an important role in our modification of the algorithm  from
\cite[Remark 2.4]{gL97}. Given the module $M$ (say through generators and relations), it follows from the proofs of Lemma \ref{associatedprimes} and Corollary \ref{maximum} that the integer $u=u(M)$ is algorithmically computable; we are leaving the details to the interested reader.

\section{The Algorithm} \label{S:theAlgorithm}
In this section we complete the description of our modification of the algorithm from \cite[Remark 2.4]{gL97} for deciding the vanishing of local cohomology modules $H^i_{\bar I}(\bar R)$.
Recall that $R=\mathbb Z[x_1,\dots,x_n]$ and $\bar R=R/pR=(\mathbb Z/p\mathbb Z)[x_1,\dots,x_n]$ where $p$ is a prime integer. Let $f_,\dots,f_s\in R$ be polynomials and see Section 1 for the definition of the Koszul cocomplex $K^\bullet(R;f_1,\dots,f_s)$. In this section, $M$ denotes the $i$-th cohomology module of $K^\bullet(R;f_1,\dots,f_s)$. Clearly $M$ is a finitely generated $R$-module. We assume that the prime integer $p$ has the property that the $i$-th cohomology module of the Koszul cocomplex  $K^\bullet(\bar R;\bar f_1,\dots,\bar f_s)$ is $\bar M=M/pM$ where $\bar f_t\in \bar R$ is the polynomial obtained from $f_t$ by reducing its coefficients modulo $p$. According to Proposition \ref{modp}, all but finitely many prime integers $p$ have this property. Let $I=(f_1,\dots,f_s)\subset R$ (resp. $\bar I=(\bar f_1,\dots,\bar f_s)\subset \bar R$) be the ideal generated by $f_1,\dots, f_s$ (resp. $\bar f_1,\dots, \bar f_s$). 

As is pointed out in Section 1, the $i$-th cohomology module of $K^\bullet(\bar R;\bar f_1^{p},\dots,\bar f_s^{p})$ is $F^{*}(\bar M)$ and $H^i_{\bar I}(\bar R)$ is the $F$-finite module generated by the map $\beta:\bar M\to F^{*}(\bar M)$ which is the map induced on cohomology by the chain map $$\beta^\bullet:K^\bullet(\bar R;\bar f_1,\dots,\bar f_s){\to} F^*(K^\bullet(\bar R;\bar f_1,\dots,\bar f_s))\cong K^\bullet(\bar R;\bar f_1^{p},\dots,\bar f_s^{p})$$ which sends $\bar R_{v_1\dots,v_i}\subseteq K^i(\bar R;\bar f_1,\dots,\bar f_s)$ to $\bar R_{v_1\dots,v_i}\subseteq K^i(\bar R;\bar f_1^p,\dots,\bar f_s^p)$ via multiplication by $(\bar f_{v_1}\cdots \bar f_{v_i})^{p-1}$. Similarly, for every $j,$ the $i$-th cohomology module of $K^\bullet(\bar R;\bar f_1^{p^j},\dots,\bar f_s^{p^j})$ is $F^{*^j}(\bar M)$ and the map $\beta_j:\bar M\to F^{*^j}(\bar M)$ of Proposition \ref{PA} is the map induced on cohomology by the chain map $$\beta_j^\bullet:K^\bullet(\bar R;\bar f_1,\dots,\bar f_s){\to} F^{*^j}(K^\bullet(\bar R;\bar f_1,\dots,\bar f_s))\cong K^\bullet(\bar R;\bar f_1^{p^j},\dots,\bar f_s^{p^j})$$ which sends $\bar R_{v_1\dots,v_i}\subseteq K^i(\bar R;\bar f_1,\dots,\bar f_s)$ to $\bar R_{v_1\dots,v_i}\subseteq K^i(\bar R;\bar f_1^{p^j},\dots,\bar f_s^{p^j})$ via multiplication by $(\bar f_{j_1}\cdots \bar f_{j_i})^{p^j-1}$. This is because $\beta_j^\bullet=F^{*^{j-1}}(\beta^\bullet)\circ\cdots\circ F^{*}(\beta^\bullet)\circ{\beta^\bullet}$, where every $F^{*^{t}}(\beta^\bullet)$ sends $\bar R_{v_1\dots,v_i}\subseteq K^i(\bar R;\bar f_1,\dots,\bar f_s)$ to $\bar R_{v_1\dots,v_i}\subseteq K^i(\bar R;\bar f_1^{p^t},\dots,\bar f_s^{p^t})$ via multiplication by $(\bar f_{j_1}\cdots \bar f_{j_i})^{(p-1)p^t}$ and equality $(p-1)+(p-1)p+(p-1)p^2+\cdots+(p-1)p^{j-1}=p^j-1$ holds.

The main result of this section is an algorithm to decide for a fixed $j$ whether \hbox{$\beta_j:\bar M\to F^{*j}(\bar M)$} is the zero map, the point being that this algorithm avoids deciding membership in an ideal generated by polynomials whose degrees rapidly grow with the growth of $p$. As a result, the memory consumed by this algorithm grows slowly with the growth of $p$ (more precisely, it grows linearly rather than exponentially). This algorithm plays a crucial role in our modification of the algorithm from \cite[Remark 2.4]{gL97}.

Denote the multi-index
$(i_1,\cdots,i_n)$
by $\bar{i}.$
Let $F^\ell:\bar R_s\to \bar R_t$ be the $\ell$-fold Frobenius homomorphism where, as in Section 1, $R_s$ and $R_t$ are copies of $R$. Since $\bar {\mathbb Z}$ is perfect, $\bar R_t$
is a free $\bar R_s$-module
on the $p^{\ell n}$ monomials $e_{\bar{i}}=x^{i_1}_1\cdots x^{i_n}_n$
where \hbox{$0\leqslant i_j<p^\ell$} for every $j.$ Suppose $N'$
is an $\bar R_s$-module. Then the pull-back $F^{*^\ell}(N')=\bar R_t\otimes_{\bar R_s}
N'=\displaystyle\bigoplus_{\overline{i}} e_{\overline{i}}\otimes_{\bar R_s} N'$ is an
$\bar R_t$-module, where $e_{\overline{i}}\otimes_{\bar R_s} N'(\cong N')$ will be called
the
$e_{\overline{i}}\:$-component of $F^*(N').$ Suppose $N''$ is an $\bar R_t$-module. For
each \hbox{$f\in \text{Hom}_{\bar R_t}(N'',F^{*^\ell}(N')),$}
define
$f_{\bar{i}}=p_{\bar{i}}\circ f:F_*^\ell(N'')\rightarrow N',$ where
$$p_{\bar{i}}: F^{*^\ell}(N')(=\bigoplus_{\bar{i}}(e_{\bar{i}}\otimes_{\bar R_s}
N'))\xrightarrow{
y\mapsto
e_{\bar{i}}\otimes p_{\bar{i}}(y)}
e_{\bar{i}}\otimes_{\bar R_s}N' (\cong N')$$ is the natural projection onto the
$e_{\overline{i}}\:$-component. We will need the following result from \cite{gL09}.

\begin{theorem}\label{TD} (Theorem 3.3 in~\cite{gL09}) We denote the multi-index $(p^\ell-1,\cdots,p^\ell-1)$ by $\overline{p^\ell-1}$. For every $\bar R_t$-module $N''$ and every $\bar R_s$-module $N',$
there is an
$\bar R_t$-linear isomorphism
\begin{align*}
\text{Hom}_{\bar R_s}(F^\ell_*(N''),N') &\cong \text{Hom}_{\bar R_t}(N'',
F^{*^\ell}(N'))\\
g_{\overline{p^\ell-1}}(-)&\leftarrow
(g=\oplus_{\bar{i}}(e_{\bar{i}}\otimes_{\bar R_s}g_{\bar{i}}(-)))\\
h&\mapsto\oplus_{\bar{i}}(e_{\bar{i}}\otimes_{\bar R_s}h(e_{\overline{p^\ell-1}-\bar{i}}
(-))).
\end{align*}
\end{theorem}

\begin{definition} \label{D:DualityOfMap}
Let $\beta_j:\bar M\to F^{*^j}(\bar M)$ be the map from Proposition \ref{PA}. Setting $N'=N''=\bar M,$ we denote by $\alpha_j:F_*^j(\bar M)\to \bar M$ the map associated to $\beta_j$ by the isomorphism in Theorem \ref{TD}, namely, $\alpha_j=(\beta_j)_{\overline{p^j-1}}$. 
\end{definition}

Theorem \ref{TD} implies the following.

\begin{corollary}\label{Cdualmap}

(a) In the above notation, $\beta_j=0$ if and only if $\alpha_j=0$. 

(b) Let $m_1,\dots,m_v$ generate $\bar M$ as an $R$-module. The map $\beta_j=0$ if and only if $\alpha_j(x^{i_1}_1\cdots x^{i_n}_nm_t)=0$ for every $t$ and every $(i_1,\cdots,i_n)$ where $0\leq i_q\leq p^j-1$ for every $q$.
\end{corollary}

\begin{proof} 
(a) is immediate from the fact that an isomorphism sends zero to zero while (b) follows from (a) and the fact that the set of elements $\{x^{i_1}_1\cdots x^{i_n}_nm_t\}$ generates $F_*^j(\bar M)$ as an $R_s$-module, so $\alpha_j=0$ if and only if $\alpha_j$ sends every generator of $F_*^j(\bar M)$ to zero.
\end{proof}

Theorem \ref{TD} admits the following straightforward extension to complexes.

\begin{corollary}\label{TDcomplexes}
For every complex of $\bar R_t$-modules $\mathcal N''^\bullet$ and for every complex of $\bar R_s$-modules $\mathcal N'^\bullet,$ there is an $\bar R_t$-linear  isomorphism
\begin{align*}
\text{Hom}_{\bar R_s}(F^\ell_*(\mathcal N''^\bullet),\mathcal N'^\bullet) &\cong \text{Hom}_{\bar R_t}(\mathcal N''^\bullet,
F^{*^\ell}(\mathcal N'^\bullet))\\
g^\bullet_{\overline{p^\ell-1}}(-)&\leftarrow
(g^\bullet=\oplus_{\bar{i}}(e_{\bar{i}}\otimes_{\bar R_s}g^\bullet_{\bar{i}}(-)))\\
h^\bullet&\mapsto\oplus_{\bar{i}}(e_{\bar{i}}\otimes_{\bar R_s}h^\bullet(e_{\overline{p^\ell-1}-\bar{i}}
(-))),
\end{align*} where $Hom$ denotes chain maps.
\end{corollary}

A chain map $g^\bullet: \mathcal N''^\bullet\to F^{*^\ell}(\mathcal N'^\bullet)$ induces a map $$g^i:H^i(\mathcal N''^\bullet)\to H^i(F^{*^\ell}(\mathcal N'^\bullet))\cong F^{*^\ell}(H^i(\mathcal N'^\bullet))$$ on cohomology where the isomorphism follows from the fact that $F^*$ is an exact functor. Let $h^\bullet:F_*^\ell(\mathcal N''^\bullet)\to \mathcal N'^\bullet$ be the chain map that corresponds to $g^\bullet$ under the isomorphism of Corollary \ref{TDcomplexes}. The chain map $h^\bullet$ induces a map $$h^i:H^i(F^\ell_*(\mathcal N''^\bullet))\cong F_*^\ell(H^i(\mathcal N''^\bullet))\to H^i(\mathcal N'^\bullet)$$ on cohomology where the isomorphism follows from the fact that $F_*$ is an exact functor. It is straightforward from the definitions and the exactness of the functors $F^*$ and $F_*$ that $h^i$ is the map associated to the map $g^i$ by the isomorphism of Theorem \ref{TD} (upon setting $N''=H^i(\mathcal N''^\bullet)$ and $N'=H^i(\mathcal N'^\bullet)$).

Let $$\alpha_j^\bullet: F_*^j(K^\bullet(\bar R;\bar f_1,\dots,\bar f_s))\to K^\bullet(\bar R;\bar f_1,\dots,\bar f_s)$$ be the chain map associated to the above chain map $\beta_j^\bullet$ by the isomorphism of Corollary \ref{TDcomplexes}. 
It follows that the map $\alpha_j:F_*^j(\bar M)\to \bar M$ induced on cohomology by the chain map $\alpha_j^\bullet$ is precisely the map associated to $\beta_j:\bar M\to F^{*^j}(\bar M)$ by the isomorphism of Theorem \ref{TD}. Thus to compute $\alpha_j(m)\in \bar M$ for some $m\in F^j_*(\bar M),$ one can take a cocycle $\tilde m\in F_*^j(K^i(\bar R;\bar f_1,\dots,\bar f_s))$ that represents $m\in F^j_*(\bar M)$, compute its image in $K^i(\bar R;\bar f_1,\dots,\bar f_s)$ via the chain map $\alpha_j^\bullet$ and take the class of this image in the $i$-th cohomology of $K^i(\bar R;\bar f_1,\dots,\bar f_s)$, i.e., in $\bar M$. This class would be $\alpha_j(m)$.

Let $m_1,\dots,m_v\in M$ generate $M$ as an $R$-module. Let $\bar m_t\in \bar M=M/pM$ be the image of $m_t$ under the natural map $M\to M/pM$. Clearly, $\bar m_1\dots, \bar m_v$ generate $\bar M$ as an $\bar R$-module. According to Corollary \ref{Cdualmap}, the map $\beta_j$ is the zero map if and only if $\alpha_j(x^{i_1}_1\cdots x^{i_n}_n\bar m_t)=0$ for every $t\leq v$ and every $(i_1,\cdots,i_n)$ where $0\leq i_q\leq p^j-1$ for every $q$. Pick a linear ordering of all the $(n+1)$-tuples $(i_1,\dots, i_n, t)$ in such a way that every tuple determines the next tuple in the ordering (i.e. no additional information is required to determine the next tuple). For example, one can order all these tuples lexicographically. Our algorithm consists in deciding, for every tuple $(i_1,\dots, i_n, t)$, whether or not $\alpha_j(x^{i_1}_1\cdots x^{i_n}_n\bar m_t)=0$. If $\alpha_j(x^{i_1}_1\cdots x^{i_n}_n\bar m_t)\ne 0$ for some tuple, the algorithm stops and returns the answer that $\beta_j$ does not vanish. If $\alpha_j(x^{i_1}_1\cdots x^{i_n}_n\bar m_t)=0$, the algorithm moves to the next tuple in the ordering and all the information about the calculations for the preceding tuple is erased from memory (it is not used the subsequent calculations). There are only finitely many tuples to consider, so the algorithm eventually stops. If a tuple $(i_1,\dots, i_n, t)$ with $\alpha_j(x^{i_1}_1\cdots x^{i_n}_n\bar m_t)\ne 0$ is never encountered, the algorithm reports that $\beta_j=0$. Thus the algorithm computes whether or not $\alpha_j(x^{i_1}_1\cdots x^{i_n}_n\bar m_t)=0$ one tuple $(i_1,\dots, i_n, t)$ at a time and the memory it consumes (modulo some finite amount that does not depend on the prime integer $p$ and is required to store the generators $m_1,\dots,m_v$ of $M$ and the current tuple $(i_1,\dots, i_n, t)$) is the memory required to decide whether or not $\alpha_j(x^{i_1}_1\cdots x^{i_n}_n\bar m_t)=0$ for just one individual tuple $(i_1,\dots, i_n, t)$.

The above considerations reduce the problem of deciding whether the map $\beta_j$ vanishes to deciding for a fixed tuple $(i_1,\dots, i_n, t)$ whether $\alpha_j(x^{i_1}_1\cdots x^{i_n}_n\bar m_t)=0$. Let $\tilde m_t\in K^i(R; f_1,\dots, f_s)$ be a cocycle that represents $m_t$ in the $i$-th cohomology module of $K^\bullet(R; f_1,\dots, f_s)$, i.e., in $M$. Let $\overline{\tilde m_t}\in K^i(\bar R; \bar f_1,\dots, \bar f_s)$ be the image of $\tilde m_t$ via the natural map $K^i(R; f_1,\dots, f_s)\to K^i(R; f_1,\dots, f_s)/pK^i(R; f_1,\dots, f_s)\cong K^i(\bar R; \bar f_1,\dots, \bar f_s).$ Clearly, $\overline{\tilde m_t}\in K^i(\bar R; \bar f_1,\dots, \bar f_s)$ is a cocycle that represents $\bar m_t$ in the $i$-th cohomology module of $K^\bullet(\bar R; \bar f_1,\dots, \bar f_s)$, i.e., in $\bar M$. As has been explained above, $\alpha_j(x^{i_1}_1\cdots x^{i_n}_n\bar m_t)\in \bar M$ is the element of $\bar M$, the $i$-th cohomology module of $K^i(\bar R;\bar f_1,\dots, \bar f_s)$, represented by the cocycle $\alpha^\bullet_j(x^{i_1}_1\cdots x^{i_n}_n\overline{\tilde m_t})\in K^i(\bar R; \bar f_1,\dots, \bar f_s).$ Thus the problem of deciding whether $\alpha_j(x^{i_1}_1\cdots x^{i_n}_n\bar m_t)=0$ reduces to first computing the cocycle $\alpha^\bullet_j(x^{i_1}_1\cdots x^{i_n}_n\overline{\tilde m_t})\in K^i(\bar R; \bar f_1,\dots, \bar f_s)$ and then deciding whether this cocycle represents the zero element in the cohomology module, i.e., whether this cocycle is a coboundary.

The module $K^i(\bar R;\bar f_1,\dots,\bar f_s)$ is a direct sum of copies of the module $\bar R$ indexed by ordered tuples $\{v_1,\dots,v_i\}$. The map $\alpha_j^\bullet$ is diagonal with respect to this direct sum decomposition, i.e., the image of $F_*^j(R_{v_1,\dots,v_i})\subseteq F_*^j(K^i(\bar R;\bar f_1,\dots,\bar f_s))$ via this map is in $\bar R_{v_1,\dots,v_i}\subseteq K^i(\bar R;\bar f_1,\dots,\bar f_s)$. In other words, the $i$-th component of the chain map $\alpha_j^\bullet$ is the direct sum of maps $\alpha_{j,v_1,\dots,v_i}:F_*^j(\bar R_{v_1,\dots,v_i})\to \bar R_{v_1,\dots,v_i}$, one map for each tuple $\{v_1,\dots, v_i\}$. Let $\overline{\tilde m}_{t,v_1,\dots,v_i}\in F_*^j(\bar R_{v_1,\dots,v_i})$ be the component of $\overline{\tilde m_t}$ in $F_*^j(\bar R_{v_1,\dots,v_i})$. The $\bar R_{v_1,\dots,v_i}$-component of the element $\alpha_j^\bullet(x^{i_1}_1\cdots x^{i_n}_n\overline{\tilde m_t})$ of $K^i(\bar R; \bar f_1,\dots, \bar f_s)$ is $\alpha_{j, v_1,\dots,v_i}(x^{i_1}_1\cdots x^{i_n}_n\overline{\tilde m}_{t, v_1,\dots, v_i})$. Thus in order to compute $\alpha_j^\bullet(x^{i_1}_1\cdots x^{i_n}_n\overline{\tilde m_t}),$ it is enough to compute $\alpha^\bullet_{j, v_1,\dots,v_i}(x^{i_1}_1\cdots x^{i_n}_n\overline{\tilde m}_{t, v_1,\dots, v_i})$ for every tuple $\{v_1,\dots,v_i\}$. The number of tuples $\{v_1,\dots,v_i\}$ is finite and does not depend on $p$. We are going to describe an algorithm to compute $\alpha^\bullet_{j, v_1,\dots,v_i}(x^{i_1}_1\cdots x^{i_n}_n\overline{\tilde m}_{t, v_1,\dots, v_i})$ for a fixed tuple $\{v_1,\dots,v_i\}$.

The map $\alpha^\bullet_{j,v_1,\dots,v_i}:F_*^j(\bar R_{v_1,\dots,v_i})\to \bar R_{v_1,\dots,v_i}$ is the map associated via Theorem \ref{TD} to the map $\beta^\bullet_{j,v_1,\dots,v_i}:\bar R\cong \bar R_{v_1,\dots,v_i}\to F^{*^j}(\bar R_{v_1,\dots,v_i})\cong \bar R$ which is nothing but the multiplication by $(f_{v_1}\cdots f_{v_i})^{p^j-1}$, as has been explained near the beginning of this section. Now for an element $y\in \bar R_{v_1,\dots,v_i}$ write $\beta^\bullet_{j,v_1,\dots,v_i}(y)=y(f_{v_1}\cdots f_{v_i})^{p^j-1}$ as $\bigoplus_{\bar i}e_{\bar i}g_{\bar i}^{p^j}$ where $g_{\bar i}\in \bar R_{v_1,\dots,v_i}$ and $e_{\bar{i}}=x^{i_1}_1\cdots x^{i_n}_n$
with \hbox{$0\leqslant i_j<p^j$} for every $j$ (every polynomial in $\bar R$ may be uniquely written in this way). By definition, $\alpha^\bullet_{j, v_1,\dots,v_i}(y)=g_{\overline{p^j-1}}.$ Setting $y=x^{i_1}_1\cdots x^{i_n}_n\overline{\tilde m}_{t, v_1,\dots, v_i}$ in this description, one gets $\alpha^\bullet_{j, v_1,\dots,v_i}(x^{i_1}_1\cdots x^{i_n}_n\overline{\tilde m}_{t, v_1,\dots, v_i})$.

Both $(f_{v_1}\cdots f_{v_i})$ and $\overline{\tilde m}_{t, v_1,\dots, v_i}$ are polynomials in $x_1,\dots, x_n$ with coefficients in $\mathbb Z/p\mathbb Z$. Let $\mathfrak m_1,\dots \mathfrak m_t$ and $\mu_1,\dots, \mu_u$ be the monomials in $x_1,\dots, x_n$ with coefficients in $\mathbb Z/p\mathbb Z$ that appear in $(f_{v_1}\cdots f_{v_i})$ and in $\overline{\tilde m}_{t, v_1,\dots, v_i}$ respectively, that is $f_{v_1}\cdots f_{v_i}=\mathfrak m_1+\dots+\mathfrak m_t$ and $\overline{\tilde m}_{t, v_1,\dots, v_i}=\mu_1+\dots+\mu_u$. Every monomial $x^{i_1}_1\cdots x^{i_n}_n\mathfrak m_1^{q_1}\cdots \mathfrak m_t^{q_t}\mu_\tau$ may be written as a monomial in the variables, i.e., $\mathfrak m_1^{q_1}\cdots \mathfrak m_t^{q_t}\mu_\tau=cx_1^{\ell_1}\cdots x_n^{\ell_n}$ where $c\in \mathbb Z/p\mathbb Z$. Define the monomial $\gamma(\mathfrak m_1^{q_1}\cdots \mathfrak m_t^{q_t}\mu_\tau)=\gamma(cx_1^{\ell_1}\cdots x_n^{\ell_n})$ as follows: $\gamma(\mathfrak m_1^{q_1}\cdots \mathfrak m_t^{q_t}\mu_\tau)=0$ if $\ell_s$ is not congruent to $p^j-1$ modulo $p^j$ for some $s$ and $\gamma(\mathfrak m_1^{q_1}\cdots \mathfrak m_t^{q_t}\mu_\tau)=cx_1^{w_1}\cdots x_n^{w_n}$ where each $w_s=\frac{\ell_t-(p^j-1)}{p^j}$ otherwise. With this notation $\alpha^\bullet_{j, v_1,\dots,v_i}(x^{i_1}_1\cdots x^{i_n}_n\overline{\tilde m}_{t, v_1,\dots, v_i})$ equals the summation of $\gamma(x^{i_1}_1\cdots x^{i_n}_n\mathfrak m_1^{q_1}\cdots \mathfrak m_t^{q_t}\mu_\tau)$ over all the monomials $\mathfrak m_1^{q_1}\cdots \mathfrak m_t^{q_t}\mu_\tau$ of total degree $q_1+\cdots+q_t=p^j-1$.

The algorithm we have consists in going through all the monomials $\mathfrak m_1^{q_1}\cdots \mathfrak m_t^{q_t}\mu_\tau $ of total degree $q_1+\cdots+q_t=p^j-1$, computing $\gamma(m_1^{q_1}\cdots m_t^{q_t}\mu_j)$ for each of them and taking their sum. More precisely, pick a well-ordering of all the monomials $\mathfrak m_1^{q_1}\cdots \mathfrak m_t^{q_t}\mu_\tau $ of total degree $q_1+\cdots+q_t=p^j-1$ in such a way that every monomial determines the next monomial in the well-ordering (i.e. no additional information is required to determine the next monomial). For example one can order all these monomials lexicographically. Dedicate a section of the memory to record partial sums of the $\gamma(m_1^{q_1}\cdots m_t^{q_t}\mu_\tau)$s. Once another $\gamma(m_1^{q_1}\cdots m_t^{q_t}\mu_\tau)$ is computed, it is added to the old partial sum and stored in its place, while the old partial sum is erased. We perform this step for each monomial $\mathfrak m_1^{q_1}\cdots \mathfrak m_t^{q_t}\mu_\tau$ in the well-ordering. Once one step is completed, we move on to the next step by passing to the next monomial in the well-ordering. The computation is completed when all the monomials $\mathfrak m_1^{q_1}\cdots \mathfrak m_t^{q_t}\mu_\tau$ in the well-ordering are exhausted. This completes the description of the computation of $\alpha^\bullet_{j, v_1,\dots,v_i}(x^{i_1}_1\cdots x^{i_n}_n\overline{\tilde m}_{t, v_1,\dots, v_i})$.

The next step in the algorithm is deciding whether the cocycle $\alpha^\bullet_{j}(x^{i_1}_1\cdots x^{i_n}_n\overline{\tilde m}_{t})$ represents the zero element in cohomology, i.e., whether this cocycle is a coboundary. Since $$\alpha^\bullet_{j}(x^{i_1}_1\cdots x^{i_n}_n\overline{\tilde m}_{t})=\bigoplus_{v_1,\dots,v_i}\alpha^\bullet_{j, v_1,\dots,v_i}(x^{i_1}_1\cdots x^{i_n}_n\overline{\tilde m}_{t, v_1,\dots, v_i})$$ and we have shown how to compute $\alpha^\bullet_{j, v_1,\dots,v_i}(x^{i_1}_1\cdots x^{i_n}_n\overline{\tilde m}_{t, v_1,\dots, v_i})$ for all ordered tuples $\{v_1,\dots, v_i\},$ standard techniques can be used to accomplish this task. Finally, if the map $\alpha_j$ is a zero map, then the map $\beta_j$ is a zero map by Corollary \ref{Cdualmap}. This completes the description of the algorithm for deciding whether the map $\beta_j:\bar M\to F^{*^j}(\bar M)$, for a fixed $j$, is the zero map.

Next we discuss the amount of memory required to perform this algorithm. The computation of the cocycle $\displaystyle \alpha^\bullet_{j}(x^{i_1}_1\cdots x^{i_n}_n\overline{\tilde m}_{t})=\bigoplus_{v_1,\dots,v_i}\alpha^\bullet_{j, v_1,\dots,v_i}(x^{i_1}_1\cdots x^{i_n}_n\overline{\tilde m}_{t, v_1,\dots, v_i})$ and deciding whether this cocycle is a coboundary for a fixed element $x^{i_1}_1\cdots x^{i_n}_n\overline{\tilde m}_{t}$ are independent of such computations for all other elements $x^{i'_1}_1\cdots x^{i'_n}_n\overline{\tilde m}_{t'}$ and the only information from one such computation that could be needed for the continuation of the algorithm is the string $(i_1,\dots, i_n,t)$. 

The computation of the element $\alpha^\bullet_{j, v_1,\dots,v_i}(x^{i_1}_1\cdots x^{i_n}_n\overline{\tilde m}_{t, v_1,\dots, v_i})$ consists of a sequence of steps, one step for each monomial $\mathfrak m_1^{q_1}\cdots \mathfrak m_t^{q_t}\mu_\tau$ of total degree $q_1+\cdots+q_t=p^j-1$, as explained above. The arithmetic operations one has to perform are the same in every step and the information that has to be kept in memory after performing one step is the string $(q_1,\dots, q_t,\tau)$ and the partial sum of the $\gamma(x^{i_1}_1\cdots x^{i_n}_n\mathfrak m_1^{q_1}\cdots \mathfrak m_t^{q_t}\mu_\tau)$s.  Each $\gamma(x^{i_1}_1\cdots x^{i_n}_n\mathfrak m_1^{q_1}\cdots \mathfrak m_t^{q_t}\mu_\tau)$, if non-zero, is a polynomial in $x_1,\dots, x_n$ of degree $$\frac{\sum_si_s+\sum_sq_s{\rm deg}\mathfrak m_s+{\rm deg}\mu_\tau-(p^j-1)n}{p^j}$$ 
Setting $D={\rm max\ deg}\mu_\tau$ and $d={\rm deg}(f_{v_1}\cdots f_{v_i})$ and taking into account that $\sum_si_s\leq (p^j-1)n$ and $\sum_sq_s{\rm deg}\mathfrak m_s\leq d({p^j-1}),$ the above fraction is bounded above by $$\frac{d(p^j-1)+D}{p^j}=d+\frac{D-d}{p^j}\leq {\rm max}\{D,d\},$$ which is a constant independent of $p$ and of the string $(i_1,\dots, i_n,q_1,\dots, q_t,\tau)$. Thus each $\gamma(x^{i_1}_1\cdots x^{i_n}_n\mathfrak m_1^{q_1}\cdots \mathfrak m_t^{q_t}\mu_\tau)$ and hence each partial sum of these is a polynomial whose degree is bounded above by a constant independent of $p$. Thus the amount of memory required to compute $\gamma(x^{i_1}_1\cdots x^{i_n}_n\mathfrak m_1^{q_1}\cdots \mathfrak m_t^{q_t}\mu_\tau)$ and memorize the resulting partial sum grows only inasmuch as one needs to store bigger and bigger coefficients of the polynomial which is the partial sum (the number of coefficients doesn't grow because the degree doesn't grow). These coefficients are elements of $\mathbb Z/p\mathbb Z$ and the amount of memory required to store those coefficients grows linearly with respect to $p$. Hence the amount of memory required to compute the element $\alpha^\bullet_{j, v_1,\dots,v_i}(x^{i_1}_1\cdots x^{i_n}_n\overline{\tilde m}_{t, v_1,\dots, v_i})$ grows linearly with respect to $p$.

The cocycle $\displaystyle \alpha^\bullet_{j}(x^{i_1}_1\cdots x^{i_n}_n\overline{\tilde m}_{t})=\bigoplus_{v_1,\dots,v_i}\alpha^\bullet_{j, v_1,\dots,v_i}(x^{i_1}_1\cdots x^{i_n}_n\overline{\tilde m}_{t, v_1,\dots, v_i})$ is an element of $K^i(\bar R;\bar f_{1},\dots, \bar f_{s})$ whose component in $\bar R_{v_1,\dots, v_i}$ is a polynomial of degree bounded above by a constant independent of $p$ and of the string $(i_1,\dots, i_n)$. The modules $K^i(\bar R;\bar f_{1},\dots, \bar f_{s})$ are free $\bar R$-modules of finite rank and the entries of the matrices defining the differentials in $K^\bullet(\bar R;\bar f_{1},\dots, \bar f_{s})$ are polynomials of $\bar R$ whose degrees do not increase with $p$. Thus the number of arithmetic operations one has to perform in order to decide whether the cocycle $\alpha^\bullet_{j}(x^{i_1}_1\cdots x^{i_n}_n\overline{\tilde m}_{t})$ is a coboundary does not increase with $p$. Hence the amount of memory required to decide whether the cocycle $\alpha^\bullet_{j}(x^{i_1}_1\cdots x^{i_n}_n\overline{\tilde m}_{t})$ is a coboundary grows only inasmuch as one needs to store bigger and bigger elements of the field $\mathbb Z/p\mathbb Z$ that appear in those arithmetic operations. The amount of memory required to store elements of $\mathbb Z/p\mathbb Z$ grows linearly with respect to $p$. This, finally, shows that the amount of memory required to decide whether the map $\beta_j:\bar M\to F^{*^j}(\bar M)$ vanishes grows linearly with respect to $p$.

Needless to say, the above algorithm is far from being practical. Even though the required memory grows only linearly, the number of arithmetic operations one has to perform grows very rapidly. This is because the same arithmetic operations have to be performed for every monomial $x_1^{i_1}\cdots x_n^{i_n}$ with $i_t\leq p^j-1$ and every monomial $\frak m_1^{q_1}\cdots\frak m_t^{q_t}$ with $\sum_tq_t=p^j-1$. The number of these monomials grows very rapidly with $p$ making the time required to complete the computation astronomical.

In conclusion we briefly summarize our modification of the algorithm  from
\cite[Remark 2.4]{gL97} for deciding the vanishing of the local cohomology module $H^i_{\bar I}(\bar R)$ where $\bar I=(\bar f_1,\dots,\bar f_s)$. First one computes the integer $u=u(M)$ (as defined in Corollary \ref{upperboundonr}) where $M$ is the $i$-th cohomology module of the Koszul complex $K^\bullet (R;f_{1},\dots, f_{s})$. According to Corollary \ref{upperboundonr} the local cohomology module $H^i_{\bar I}(\bar R)$ (which is the $F$-finite module generated by the map $\beta:\bar M\to F^*(\bar M)$) vanishes if and only if the map $\beta_u:\bar M\to F^{*^u}(\bar M)$ is the zero map. Thus all one has to do is apply our algorithm for deciding whether the map $\beta_j:\bar M\to F^{*^j}(\bar M)$ vanishes for $j=u$.

\begin{acknowledgement} 
Part of this work is from the author's dissertation. The author gratefully thanks his advisor Professor Gennady Lyubeznik for his continued support and guidance.
\end{acknowledgement}

\end{document}